\documentclass{amsart}
\usepackage{amsfonts}
\usepackage{amsmath,amscd}
\usepackage{amssymb}
\usepackage{amsthm}
\usepackage{newlfont}

\newcommand{\ds}{\displaystyle}

\long\def\alert#1{\parindent2em\smallskip\hbox to\hsize%?
{\hskip\parindent\vrule%?
\vbox{\advance\hsize-2\parindent\hrule\smallskip\parindent.4\parindent%?
\narrower\noindent#1\smallskip\hrule}\vrule\hfill}\smallskip\parindent0pt}
 \newtheorem{thm}{Theorem}[section]
\newtheorem{cor}[thm]{Corollary}
 \newtheorem{lem}[thm]{Lemma}
 \newtheorem{prop}[thm]{Proposition}
\theoremstyle{definition}
 
\theoremstyle{remark}

 \theoremstyle{problem}
 
 \numberwithin{equation}{section}
\newtheorem*{Theorem A}{\textbf{  Theorem A}}
\newtheorem*{Main Theorem}{\textbf{Main Theorem}}
\newtheorem*{p a}{\textbf{Proof of  Theorem A}}
\newtheorem*{p b}{\textbf{Proof of the Main Theorem}}
\begin{document}

\title[Order of the Schur multiplier of $p$-groups]
 {Classification of finite $p$-groups by the size of their  Schur multipliers}
\author[P. Niroomand]{Peyman Niroomand}
\address{School of Mathematics and Computer Science\\
Damghan University, Damghan, Iran}
\email{niroomand@du.ac.ir, p$\_$niroomand@yahoo.com}
\author[F. Johari]{Farangis Johari}
\address{ Departamento de Matem\'{a}tica, Instituto de Ci\^{e}ncias Exatas, Universidade Federal de Minas Gerais, Av. Ant\^{o}nio Carlos 6627, Belo Horizonte, MG, Brazil.}
\email{farangisjohari@ufmg.br,farangisjohari85@gmail.com}



\date{\today}

%\dedicatory{}?

%?
\begin{abstract}
Let $d(G)$ be the minimum number of elements required to generated a group $G.$ For a group $G $ of order 
$p^n$ with derived subgroup of order $  p^k $ and $d(G) = d,$ we knew  the order of the Schur multiplier of $G$ is bounded by $ p^{\frac{1}{2}(d-1)(n-k+2)+1}. $ In the current paper, we find the structure of all $p$-groups that attains
 the mentioned bound. Moreover,  we show that all of them  are capable.

\end{abstract}

%%% ----------------------------------------------------------------------
\maketitle
%%% ----------------------------------------------------------------------

\section{Motivation and Preliminaries}
The Schur multiplier, $\mathcal{M}(G),  $  of a group $G$ first appeared in $1904$ in the  work of Schur    on  projective representations of groups.

The Schur multiplier was studied by several authors  and proved to be an important tool in the classification of $p$-groups. For a  group $G$ of order $ p^n,$ by a result of Green in \cite{18a}, we have  $ |\mathcal{M}(G)|\leq p^{\frac{1}{2}n(n-1)-t(G)} $ with $t(G)\geq 0.  $ Several authors  characterized the structure of $p$-groups by using $t(G). $ The reader can find the structure
of $p$-groups when $t(G)\in \{0,\ldots,6\}  $ (see \cite{2,9,hau,ni12,salemkar,zh}).
Later, the first author \cite{25}  improved  Green's bound  and showed for any non-abelian group  $G $ of order $ p^n $  with $ |G'|=p^k, $ we have \begin{align}\label{kh1}|\mathcal{M}(G)|\leq p^{\frac{1}{2}(n-k-1)(n+k-2)+1}.\end{align}  He  also characterized  all of $p$-groups that attain the  upper bound when $ k=1. $ Recently, Rai \cite{rai1} improved  this bound. He showed  for a  $p$-group $G $ of order $ p^n $ with $ |G'|=p^k $ and $ d(G)=d,$ we have
\begin{align}\label{kh}
|\mathcal{M}(G)|\leq p^{\frac{1}{2}(d-1)(n+k-2)+1}.
\end{align}
In the present paper, we are going to find the structure of all
$p$-groups that attain  the bound \eqref{kh}, and then we show that all of them are capable.
 \\The concept of the non-abelian tensor square $G\otimes G $ of a group $G$ is a special case of the non-abelian tensor product of two arbitrary groups that was introduced by Brown and Loday  \cite{Bro}. It is easy to check that $\kappa : G\otimes G \rightarrow G'$ given by $g \otimes g'\rightarrow [g, g' ]$ for all $ g,g'\in G$ is an epimorphism. Let $J_2(G)  $ be the kernel of $ \kappa, $ and let $ \bigtriangledown(G) $ be a subgroup of $G\otimes G  $  generated by the set $\{g\otimes g\mid g\in G\}.  $ Clearly, $\bigtriangledown(G)   $ is a central subgroup of  $G\otimes G.$
The non-abelian exterior square $ G\wedge G$ is the quotient group $ \dfrac{G\otimes G}{\bigtriangledown(G)}. $ The element $ (g\otimes g') \bigtriangledown(G)$ in $ G\wedge G $ is denoted by $ g\wedge g' $ for all $ g,g'\in G. $ The map $ \kappa $ induces the epimorphism  $\kappa' : G\wedge G \rightarrow G'$ given by $g \wedge g'\rightarrow [g, g' ]$ for all $ g,g'\in G.$ The kernel of the map $\kappa'  $ is isomorphic to the Schur multiplier  of $G$ (for more information,
see \cite{Bro}).\\
Recall that a group $ G $ is called capable  provided that $ G\cong H/Z(H) $ for a group $ H. $ Beyl et al.  \cite{3} gave a
criterion for detecting capable groups. They showed  that a group $G$ is capable if and only if the epicenter of $G,$ $ Z^*(G),$ is trivial.
Ellis \cite{111} showed $Z^{\wedge}(G)=Z^*(G),$ where $Z^{\wedge}(G)$ is the exterior center of $G,$ i.e.  the set of all elements $g$ of $G$ for which
$g \wedge h = 1_{G\wedge G}$ for all $h \in  G$ 
(see for instance \cite{111}  to find  more information in  these topics).

The following technical result  characterizes the structure of all minimal non-abelian $p$-groups.
\begin{lem}\label{9}
\cite[Exercise 8a.]{ber} and \cite{red}
Let $G$ be a minimal non-abelian $p$-group. Then $|G'|= p$ and $G$ is isomorphic to one of the following groups:
\begin{itemize}
\item[$(a).$]
$ G\cong \langle a, b\mid a^{p^m}=1, b^{p^n}=1,[ a, b]= a^{p^{m-1}},[a,b,a]=[a,b,b]=1\rangle
$ for all $ m, n $ such that $ m\geq 2, n\geq 1.$ Moreover, $ |G|= p^{m+n}$ and $ Z(G)=\langle a^p\rangle \times \langle b^p\rangle $.
\item[$(b).$] $G\cong \langle a, b\mid a^{p^m}= b^{p^n}=[ a, b]^p=1,[a,b,a]=[a,b,b]=1 \rangle$ is
 of order $ p^{m+n+1} $ and if $p = 2,$ then $m + n > 2.$ Moreover, $ Z(G)=\langle a^p\rangle \times \langle b^p\rangle \times \langle [ a, b]\rangle.$
\item[$(c).$] $G\cong Q_8.$
\end{itemize}
\end{lem}
Let $ \mathbb{Z}_{n}^{(t)} $  denote the direct sum of $t$ copies of $ \mathbb{Z}_{n}, $ in which $ \mathbb{Z}_{n} $ is the cyclic group
of order $  n.$

\begin{prop}\label{min}
Let $G$ be a minimal non-abelian $p$-group as in Lemma \ref{9}$ (a), $ where $n \geq 1,$ and  $m \geq 2$ if $ p>2, $ $ m\geq 3$ if $ p=2.$  If $ G/G' $ is homocyclic, then $ n=m-1, $ $ G $ is non-capable, and $ Z^{\wedge}(G)=G'. $
\end{prop}
\begin{proof}
Since $ G/G' $ is homocyclic and $  G/G'\cong \mathbb{Z}_{p^n}\times \mathbb{Z}_{p^{m-1}}, $  $  n=m-1. $ We claim that $ Z^{\wedge}(G)=G'. $
  \cite[Corollary 7.4]{3} implies $ G/G' $ is capable and so $ Z^{\wedge}(G)\subseteq G'\cong  \mathbb{Z}_{p},$ by \cite[Corollary 2.2]{3}. It is sufficient to show  that $  G'=\langle a^{p^{m-1}}\rangle  \subseteq Z^{\wedge}(G).$ Since  $b^{^{p^{m-1}}}=1,  $ we have
\begin{align*}
1_{G\wedge G}&=b^{^{p^{m-1}}}\wedge  a=\overset{p^{m-1}-1}{\prod_{i=0}} (^{b^i}(b\wedge a))=
\overset{p^{m-1}-1}{\prod_{i=0}} (a\wedge [b^i,a]a)\\&=\overset{p^{m-1}-1}{\prod_{i=1}} (b\wedge [b, a])^i\overset{p^{m-1}-1}{\prod_{i=0}} (b\wedge a)=\big{(}
\overset{p^{m-1}-1}{\prod_{i=1}} (b\wedge [b,a])^i \big{)}(b\wedge a)^{p^{m-1}}.
\end{align*}
Clearly,
 \begin{align*}
 \overset{p^{m-1}-1}{\prod_{i=1}} (b\wedge [b,a])^i&=(b\wedge [b, a])^{\frac{1}{2}p^{m-1}(p^{m-1}-1)} \\&=b\wedge [b, a]^{\frac{1}{2}p^{m-1}(p^{m-1}-1)}
 %=b\wedge [b^{\frac{1}{2}p^{m-1}(p^{m-1}-1)}, a]. 
 \end{align*}
 Since $ [a,b]^p=1, $ we get  
 $ [b, a]^{\frac{1}{2}p^{m-1}(p^{m-1}-1)}=1,
$ and so
$ \ds\overset{p^{m-1}-1}{\prod_{i=1}} (b\wedge [b,a])^i=1. $ Thus
\[b^{p^{m-1}}\wedge  a=(b\wedge a)^{p^{m-1}}= 1_{G\wedge G}.\]
 We will show that $a^{p^{m-1}}\wedge b=(a\wedge b)^{p^{m-1}}.  $ Since
\begin{align*}
a^{p^{m-1}}\wedge b &=\overset{p^{m-1}-1}{\prod_{i=0}} (^{a^i}(a\wedge b))\\&=\overset{p^{m-1}-1}{\prod_{i=0}} (a\wedge [a^i,b ]b)\\&=\overset{p^{m-1}-1}{\prod_{i=0}} (a\wedge b)=(a\wedge b)^{p^{m-1}},
\end{align*}
 we have $a^{p^{m-1}}\wedge b=(a\wedge b)^{p^{m-1}}=(b\wedge a)^{-p^{m-1}}= (b^{p^{m-1}}\wedge  a)^{-1}=1_{G\wedge G}.$ Moreover, $ a\wedge [ a, b]= a\wedge  a^{p^{m-1}}=1_{G\wedge G}.$ Therefore $1 \neq a^{p^{m-1}}\in Z^{\wedge}(G) $ and so $ Z^{\wedge}(G)=G'.$ Hence  $ G $ is non-capable.
\end{proof}
Let $ d(G) $  denote the minimum number of elements required to generate a group $ G. $

\begin{lem}\label{lll}Let $ G $ be a capable non-abelian  $p$-group of order $ p^n $ such that $ |G'|=p, $  $ d(G)=d,$ and $ e(G/G')>p.$ Then $ |G/Z(G)|=p^2$ and $ G= NZ(G),$ where  $ N $ is a minimal non-abelian $p$-group. 
\end{lem}
\begin{proof}
 \cite[Theorem C]{isa} implies that $ |G/Z(G)|=p^2.$
Using \cite[Lemma 4.2]{ber}, we have  $ G=NZ(G), $ where $ N $ is minimal non-abelian.
\end{proof} 
Let  $ \gamma_i(G) $ be denoted the $  i$-th term of the lower central series of a group $  G.$
We need the following result in the proof of Theorem \ref{25}.
\begin{prop}\cite[Proposition 1]{ele} and \cite{el,hau}\label{jk}
Let $ G $ be a finite  non-abelian $p$-group of class $ c. $
\begin{itemize}
\item[$  (i)$]
The map \[\Psi_2:  (G/Z(G))^{(ab)}\otimes  (G/Z(G))^{(ab)}\otimes (G/Z(G))^{(ab)}  \rightarrow \big{(}G'/\gamma_3(G)\big{)}\otimes G/G'\] given by
$xG'Z(G)\otimes yG'Z(G) \otimes zG'Z(G)\mapsto $\[ ([x,y]\gamma_3(G)\otimes zG' )
([z,x]\gamma_3(G)\otimes yG')([y,z]\gamma_3(G)\otimes xG')
\]
is a  homomorphism. If any two elements of the set $\{ x,y,z\} $ are linearly dependent, then $\Psi_2(xG'Z(G) \otimes yG'Z(G) \otimes zG'Z(G))=1_{\big{(}G'/\gamma_3(G)\big{)}\otimes G/G'}.$ \item[$  (ii)$]
The map
\begin{align*}
\Psi_3: &(G/Z(G))^{(ab)}\otimes (G/Z(G))^{(ab)}\otimes (G/Z(G))^{(ab)}\otimes  (G/Z(G))^{(ab)}\\& \rightarrow  \gamma_3(G)/\gamma_4(G)\otimes  (G/Z(G))^{(ab)}~\text{given by}~\\&\big{(}x G'Z(G) \big{)}\otimes \big{(}yG'Z(G) \big{)} \otimes \big{(}z G'Z(G)\big{)}\otimes \big{(}w G'Z(G) \big{)}\\
&\mapsto\big{(}[[x,y],z]\gamma_4(G)\otimes wG'Z(G) \big{)}
\big{ (}[w,[x,y]]\gamma_4(G)\otimes zG'Z(G)\big{)}\\&\big{(}[[z,w],x]\gamma_4(G)\otimes yG'Z(G)\big{)}\big{(}[y,[z,w]]\gamma_4(G)\otimes xG'Z(G)\big{)}
\end{align*}
is a  homomorphism.
\end{itemize}
\end{prop}

\begin{thm}\label{25}
Let $ G $ be a non-abelian  group of order $ p^n $ of class $c$ with $ |G'|=p^k $ and $ d=d(G). $ Then
 \begin{align*}| G\wedge G|| \mathrm{Im}\Psi_2|| \mathrm{Im} \Psi_3| &\leq | G\wedge G|\prod_{i=2}^c | \ker \alpha_i|=| \mathcal{M}(G)| |G'|\prod_{i=2}^c | \ker \alpha_i|\\&=| \mathcal{M}(G/G')|\prod_{i=2}^c|\gamma_i(G)/\gamma_{i+1}(G)\otimes G/G'|\leq | \mathcal{M}(G/G')| p^{kd},\end{align*}
where
\[\alpha_{i}: \gamma_i(G)\wedge G \rightarrow  \gamma_{i-1}(G)\wedge G\]\[x\wedge z \mapsto x \wedge z\] is a natural homomorphism
for all $ i$ such that $ i\geq 2. $
\end{thm}
\begin{proof}
Similar to the proof of \cite[Proposition 5]{el}, we have \[| G\wedge G|\prod_{i=2}^c | \ker \alpha_i|=|G/G'\wedge G/G' | \prod_{i=2}^c|\gamma_i(G)/\gamma_{i+1}(G)\otimes G/G'|.\] Since $| G\wedge G| =| \mathcal{M}(G)| |G'|$ and using the proof of \cite[Theorem 1.2]{rai1}, we have \begin{align*} &| G\wedge G|\prod_{i=2}^c | \ker \alpha_i|=| \mathcal{M}(G)| |G'|\prod_{i=2}^c | \ker \alpha_i|\\&=| \mathcal{M}(G/G')|\prod_{i=2}^c|\gamma_i(G)/\gamma_{i+1}(G)\otimes G/G'|\leq | \mathcal{M}(G/G')| p^{kd}.
\end{align*}
We claim that $| \mathrm{Im}\Psi_2| | \mathrm{Im} \Psi_3|\leq \prod_{i=2}^c | \ker \alpha_i|.$
Clearly, the map \[\delta_i: \gamma_i(G)/\gamma_{i+1}(G)\otimes G/G' \rightarrow \gamma_i(G)/\gamma_{i+1}(G)\otimes G/Z(G)G' \] given by $ x\gamma_{i+1}(G) \otimes yG'\mapsto  x\gamma_{i+1}(G) \otimes yG'Z(G)$ is a natural epimorphism for all $ i\geq 2. $ By \cite[Proposition 2.1]{ri4}, we have
\begin{align*}
| \mathcal{M}(G)| |G'|| \mathrm{Im}\Psi_2|| \mathrm{Im} \Psi_3|& \leq | \mathcal{M}(G/G')|\prod_{i=2}^c|\gamma_i(G)/\gamma_{i+1}(G)\otimes G/G'Z(G)|\\&=| \mathcal{M}(G/G')|\prod_{i=2}^c|\gamma_i(G)/\gamma_{i+1}(G)\otimes G/G'|/|\ker \delta_i |\\&=| \mathcal{M}(G)| |G'|\prod_{i=2}^c | \ker \alpha_i |/|\ker \delta_i |\end{align*} and so $ | \mathrm{Im}\Psi_2|| \mathrm{Im} \Psi_3|\leq  \prod_{i=2}^c | \ker \alpha_i |/|\ker \delta_i | \leq \prod_{i=2}^c | \ker \alpha_i |.$
 Therefore \begin{align*}| G\wedge G|| \mathrm{Im}\Psi_2|| \mathrm{Im} \Psi_3| &\leq | G\wedge G|\prod_{i=2}^c | \ker \alpha_i|\\&=| \mathcal{M}(G)| |G'|\prod_{i=2}^c | \ker \alpha_i|\\&=| \mathcal{M}(G/G')|\prod_{i=2}^c|\gamma_i(G)/\gamma_{i+1}(G)\otimes G/G' |\\&\leq | \mathcal{M}(G/G')| p^{kd}.\end{align*}
 The proof is completed.
\end{proof}
Let $1\rightarrow R\rightarrow F \xrightarrow{\pi}  G\rightarrow 1 $ be a free presentation for a group $ G $ and the exponent of a group $ X$ is denoted by  $ e(X). $ 

The next result is extracted from the work of Blackburn and Evens in
\cite[Remark, Section 3]{black burn}.
\begin{thm}\label{lkk}
Let $ G $ be a   non-abelian $p$-group of class two and $ e(G/G')=p^{s} $. Then
%\item[$  (1)$] 
 $1\rightarrow \ker \eta \rightarrow G'\otimes G/G' \xrightarrow{\eta} \mathcal{M}(G)\rightarrow
\mathcal{M}(G/G')\rightarrow G' \rightarrow 1$
is exact, in which \[\eta:  G'\otimes G/G'\rightarrow  \mathcal{M}(G)=(R\cap F')/[R,F]\]\[x\otimes (zG')\mapsto [\tilde{x},\tilde{z}][R,F]\] such that
 $ \pi(\tilde{x}R)=x $ and $ \pi(\tilde{z}R)=z.$
Moreover,  \[\langle ([x,y]\otimes zG') 
([z,x]\otimes y G')([y,z]\otimes xG'),w^{p^s}\otimes w G'\mid x,y,z,w\in G\rangle \subseteq  \ker \eta.\]

\end{thm}
\begin{proof}
By \cite[Corollary 3.2.4]{kar}, we get \[1\rightarrow \ker \eta \rightarrow G'\otimes G/G' \xrightarrow{\eta} \mathcal{M}(G)\rightarrow
\mathcal{M}(G/G')\rightarrow G' \rightarrow 1\]
is exact, in which \[\eta:  G'\otimes G/G'\rightarrow \mathcal{M}(G)=(R\cap F')/[R,F]\] \[x\otimes zG' \mapsto [\tilde{x},\tilde{z}][R,F]\] such that
 $ \pi(\tilde{x}R)=x $ and $ \pi(\tilde{z}R)=z.$
Similar to the proof of \cite[Theorem 3.1]{black burn}, we have
$\langle ([x,y]\otimes zG')
([z,x]\otimes y G')([y,z]\otimes xG'),w^{p^s}\otimes w G'\mid x,y,z,w\in G\rangle \subseteq  \ker \eta,$ as required.
\end{proof}
\begin{lem}\label{ll}
Let $ G $ be a group of class two such that $ d(G/Z(G))=d $ is finite. Then $ d(G')\leq \frac{1}{2} d(d-1).$
\end{lem}
\begin{proof}
We can choose a generating set $ \{ x_1 Z(G),\ldots,x_d Z(G)\} $ for $ G/Z(G) $ such that $ [x_i,x_j] $ is non-trivial for $ i\neq j. $
It is clear   to see that $ \{[x_i,x_j]\mid 1\leq i<j\leq d\} $ generates $G',$ as required.
\end{proof}

\begin{lem}\label{10}
Let $ G $ be a group.
\begin{itemize}
\item[$(i)$]If $G\cong \langle a, b\mid  a^{p^m}= b^{p^m}=[ a, b]^{p^{m}}=1,[a,b,a]=[a,b,b]=1,m\geq 2 \rangle$ for $ p\neq 2, $ then  $ \mathcal{M}(G)\cong  \mathbb{Z}_{p^{m}} \times \mathbb{Z}_{p^{m}}. $
\item[$(ii)$]If $G\cong \langle a, b\mid  a^{p^m}= b^{p^m}=[ a, b]^{p^{k}}=1,[a,b,a]=[a,b,b]=1,1\leq k< m  \rangle,$  then  $ \mathcal{M}(G)\cong \mathbb{Z}_{p^{m-k}}\times \mathbb{Z}_{p^{k}} \times \mathbb{Z}_{p^{k}}. $
\end{itemize}
\end{lem}
\begin{proof}
It is clearly obtained by \cite[Theorems 49 and 50]{morse}.
\end{proof}
\begin{thm} \cite[Theorem 1.1]{rai3}\label{16}
Let $ G $ be a non-abelian  group of order $ p^n $ of class two and $ |G'|=p^k.$ Then
$|\mathcal{M}(G)|=p^{\frac{1}{2}(n-k-1)(n+k-2)+1}$ if and only if $ G $ is isomorphic to one of the following  groups:
\item[$ (1) $] For  $ p\neq 2, $ $G_1\cong E_1\times \mathbb{Z}_{p}^{(n-3)}, $ where $ E_1 $ is the extra-special $p$-group of order $ p^3 $ and exponent $ p.$ 
\item[$ (2) $] For  $ p\neq 2, $ $G_2\cong \mathbb{Z}_{p}^{(4)}\rtimes\mathbb{Z}_{p}. $
\item[$ (3) $] For  $ p\neq 2, $
\[ G_3\cong \langle x_1,x_2,x_3\mid [x_1,x_2]^p=[x_2,x_3]^p=[x_3,x_1]^p=x_i^p=1,\]\[[x_1,x_2,x_i]=[x_3,x_1,x_i]=[x_2,x_3,x_i]=1, 1\leq i\leq 3\rangle. \]
\end{thm}

\section{main results}
As proven in   \cite[Theorem 1.1]{rai1}, the Schur multiplier of a non-abelian  group $G $ of order $ p^n $ with $ |G'|=p^k $  and $ d(G)=d $ is bounded by $ p^{\frac{1}{2}(d-1)(n+k-2)+1}.  $ Let $ p $ is an odd prime number. The main result of this paper is devoted  to characterizing the structure of all finite
$p$-groups   that attain the mentioned upper bound. Moreover, we  show that all   $p$-groups that attain the bound are capable.
 Throughout the paper, we say that $ |\mathcal{M}(G)| $ attains the bound provided that $ |\mathcal{M}(G)|= p^{\frac{1}{2}(d-1)(n+k-2)+1}. $

\begin{Main Theorem}
{\em  Let $ G $ be a non-abelian  group of order $ p^n $   with $ p\neq 2, $ $ |G'|=p^k,$ and $ d(G)=d. $ Then
$ |\mathcal{M}(G)|=p^{\frac{1}{2}(d-1)(n+k-2)+1}$  if and only if $ G $ is isomorphic to one of the following  groups:
\begin{itemize}
\item[$(i)$]    $H_1\cong E_1\times \mathbb{Z}_{p}^{(n-3)}, $ where $ E_1 $ is the extra-special $p$-group of order $ p^3 $ and exponent $ p.$ 
\item[$(ii)$]$H_2\cong \langle a, b\mid a^{p^m}= b^{p^m}=[ a, b]^p=1,[a,b,a]=[a,b,b]=1 \rangle,$ where  $m  > 1.$
\item[$(iii)$] $H_3\cong \langle a, b\mid a^{p^m}= b^{p^m}=[ a, b]^{p^{m}}=1,[a,b,a]=[a,b,b]=1 \rangle,$ where   $ m\geq 2.$
\item[$(iv)  $]$H_4\cong \langle a, b\mid a^{p^m}= b^{p^m}=[ a, b]^{p^{k}}=1,[a,b,a]=[a,b,b]=1,k\geq 2,m\geq 2 \rangle.$
 \item[$(v)  $] $H_5\cong\mathbb{Z}_{p}^{(4)} \rtimes\mathbb{Z}_{p}.$ 
\item[$(vi)  $]$ H_6\cong \langle x_1,x_2,x_3\mid [x_1,x_2]^p=[x_2,x_3]^p=[x_3,x_1]^p=x_i^p=1,$\[[x_1,x_2,x_i]=[x_3,x_1,x_i]=[x_2,x_3,x_i]=1, 1\leq i\leq 3\rangle.\]  
\item[$(vii)  $] $H_7\cong\langle x_1,y_1,x_2,y_2,x_3,y_3,z\mid [x_1,x_2]=y_3,[x_2,x_3]=y_1,[x_3,x_1]=y_2,$\[[y_i,x_i]=z,[x_i,y_j]=1,x_i^3=y_i^3=z^3=1, (i=1,2,3,~j=2,3)\rangle. \]
\end{itemize}}
\end{Main Theorem}
Let $cl(X)$ be used to denote 
 nilpotency class of a group $X$.
 We begin with the following lemma for the future convenience.
\begin{lem}\label{m1}
Let $ G $ be a non-abelian  group of order $ p^n $ such that $ |G'|=p^k, $  $ d(G)=d, $ and $G/G'\cong \mathbb{Z}_{p^{\alpha_1}}\oplus \ldots\oplus \mathbb{Z}_{p^{\alpha_d}}$, where $\alpha_1 \geq \ldots \geq \alpha_d.$ If
$|\mathcal{M}(G)|$ attains the bound, then
\begin{itemize}
\item[$(i)  $]$G/G'$ is  homocyclic.
\item[$(ii)  $]  $\text{Im}~ \Psi_2\cong \mathbb{Z}_p^{(d-2)},$   $\text{Im}~ \Psi_3=1,$ and $|\text{Im}~ \Psi_2| =|\ker \alpha_2|. $
\end{itemize}
\end{lem}
\begin{proof}
\begin{itemize}
\item[$(i)  $]By using \cite[Corollary 1.3]{rai1}, we have
\begin{align*}
&|\mathcal{M}(G)|= p^{\frac{1}{2}(d-1)(n+k-2)+1}\leq  p^{\frac{1}{2}(d-1)(n+k-2-\alpha_1-\alpha_d)+1}.
\end{align*}
Thus $ \alpha_1=\alpha_d $ and so $G/G'$ is  homocyclic.

\item[$(ii)  $]Let $ cl(G)=c. $ By part $ (i), $  $G/G'$ is  homocyclic and so $ n=d\alpha_1+k. $
Theorem \ref{25} and \cite[Corollary 2.2.12]{kar} imply
 \begin{align*}
  p^{\frac{1}{2}(d-1)(n+k-2)+1}p^{k}\prod_{i=2}^c | \ker \alpha_i |&= p^{\frac{1}{2}(d-1)(d\alpha_1+2k-2)+1}p^{k}\prod_{i=2}^c | \ker \alpha_i |\\&=| \mathcal{M}(G)||G'|\prod_{i=2}^c | \ker \alpha_i |\\&= | \mathcal{M}(G/G')|\prod_{i=2}^c| \gamma_i(G)/\gamma_{i+1}(G)\otimes G/G'|\\&\leq p^{\frac{1}{2}d
 (d-1)\alpha_1}p^{kd}.
 \end{align*}
 Thus
 \begin{equation}\label{eq}
  \prod_{i=2}^c | \ker \alpha_i | \leq p^{d-2}.
 \end{equation}
 Consider
  $ G= \langle x_1,x_2,\ldots,x_d\rangle$ and $1\neq [x_1,x_2]\in G'\setminus \gamma_3(G)$.
   We claim that $\text{Im} \Psi_2\cong \mathbb{Z}_p^{(d-2)}.$ 
   Similar to the proof of  \cite[Theorem 2]{ele}, the set $A=\{\Psi_2(x_1 G'\otimes x_2 G'\otimes x_j G')\mid 3\leq j\leq d\} $ consists of $d-2$ linearly independent elements of order at least $p$ in the abelian $ p$-group $ G/G'\otimes G'/\gamma_3(G).$ Hence  $  p^{d-2}\leq |\langle A\rangle |\leq |\text{Im} \Psi_2|$.
   By using (\ref{eq}) and the proof of  Theorem \ref{25}, 
    \begin{align*}
     p^{d-2}&\leq |\langle A\rangle |\leq|\text{Im} \Psi_2|\\&\leq |\text{Im} \Psi_2||\text{Im} \Psi_3|\leq | \ker \alpha_2| \\&\leq \prod_{i=2}^c | \ker \alpha_i | \\&\leq p^{d-2}.
    \end{align*} 
     Hence
     $|\langle A\rangle |=|\text{Im} \Psi_2|=| \ker \alpha_2|=p^{d-2}$ and  $|\text{Im} \Psi_3|=1$ and so $\langle A\rangle =\text{Im} \Psi_2\cong \mathbb{Z}_p^{(d-2)}.$ 
\end{itemize}
\end{proof}

\begin{prop}\label{m4}
Let $ G $ be a non-abelian  group of order $ p^n $  with $ |G'|=p^k,$  $ d(G)=d $ and
$|\mathcal{M}(G)|=p^{\frac{1}{2}(d-1)(n+k-2)+1}.$ If $ k\geq 2$ and
$K$ is a central subgroup   of order $p$
contained in $ Z(G)\cap G',$  then $|\mathcal{M}(G/K)|$ also attains the bound, that is \[|\mathcal{M}(G/K)| =p^{\frac{1}{2}(d-1)(n + k - 4) + 1}.\]
\end{prop}
\begin{proof}
Let $K\subseteq Z(G)\cap G'$ and $ |K|=p. $
Lemma \ref{m1}$ (i) $ implies  $ G/G'\cong \mathbb{Z}^{(d)}_{p^{m}} $ for some $ m\geq 1, $ and so  $ G/G' \otimes K\cong \mathbb{Z}^{(d)}_{p}. $
Using \cite[Theorem 4.1]{17}, we have
\begin{align*}
 | \mathcal{M}(G)||G'\cap K| &\leq | \mathcal{M}(G/K) || \mathcal{M}(K)| |G/G' \otimes K|
\\&= | \mathcal{M}(G/K)||G/G' \otimes K|.
\end{align*}
 Thus
\begin{align*}
| \mathcal{M}(G)|&\leq  | \mathcal{M}(G/K) ||G/G' \otimes K|p^{-1}\\&\leq
p^{\frac{1}{2}(d- 1)(n + k - 4) + 1 + d - 1}\\
=&p^{\frac{1}{2}
(d - 1)(n + k - 2) + 1}=| \mathcal{M}(G)|.
\end{align*}
Therefore $|\mathcal{M}(G/K)|=p^{\frac{1}{2}(d- 1)(n + k - 4) + 1},$
as required.
\end{proof}
The following lemma shows that if the order of the Schur multiplier of  a non-abelian $p$-group  $ G $  attains the bound, then  $ G $ is capable.
\begin{lem}\label{ff}
Let $ G $ be a non-abelian   $p$-group of order $ p^n $  with $ |G'|=p^k$ and  $ d(G)=d. $ If
$|\mathcal{M}(G)|$ attains the bound, then $ G $ is capable.
\end{lem}
\begin{proof}
Lemma \ref{m1}$(i)$ implies $ G/G'$ is homocyclic. Therefore $ G/G' $ is capable, by \cite[Corollary 7.4]{3}. Hence  $ Z^*(G)\subseteq G',$ by \cite[Corollary 2.2]{3}. Assume to the contrary that $ G $ is non-capable. Then there is a normal subgroup $ K $ of order $p$ in $ Z^*(G).$ By using  \cite[Theorem 4.2]{3} and  \cite[Theorem 2.5.6$(i)$]{kar},   we have $ |\mathcal{M}(G)|=|\mathcal{M}(G/K)|p^{-1}.$ Proposition \ref{m4} implies that $ |\mathcal{M}(G)|= p^{\frac{1}{2}(d- 1)(n + k - 4) },$  which is a contradiction, since by our assumption, we have $ |\mathcal{M}(G)|= p^{\frac{1}{2}(d- 1)(n + k - 2)+1}.$ Thus $  Z^*(G)=1 $ and so the result follows from \cite[Corollary 2.3]{3}.
\end{proof}
%Proposition 2.4, a key step in the proof of the main theorem, is not correct.  To see this one can consider the group given in the proof of Proposition 2.4 itself (Page 8, Line 17) for m = 2. This group attains the bound and satisfies all the conditions of Proposition 2.4.  
%
%
%
%It is very strange that the authors are aware of this fact and include this group in their classification given in an earlier version of this paper present on arxiv and then remove it from the version submitted to this journal.
\begin{prop}\label{m2}
Let $ G $ be a non-abelian finite $p$-group of order $ p^n $ such that $ |G'|=p $ and $ d(G)=d. $ If  $ e(G/G')>p,$ then
$|\mathcal{M}(G)|$ attains the bound  if and only if
  $G\cong \langle a, b\mid a^{p^m}= b^{p^m}=[ a, b]^p=1,[a,b,a]=[a,b,b]=1 \rangle,$ where  $m  \geq 1.$
\end{prop}
\begin{proof}
Assume that $|\mathcal{M}(G)|$ attains the bound. By Lemmas \ref{lll} and \ref{ff}, we have 
  $ |G/Z(G)|=p^2$ and $ G= NZ(G),$ where  $ N $ is a minimal non-abelian $p$-group. Since $ G/G' $ is homocyclic, we get $ G/G'\cong \mathbb{Z}_{p^{\alpha_1}}^{(d)} $ and $ G/Z(G)\cong \mathbb{Z}_{p}\oplus \mathbb{Z}_{p}. $   Since $ G= NZ(G),$ we have $G'=N'=\langle c\rangle  $  such that $ N=\langle a,b\rangle $ and $ c=[a,b]. $ We claim that $ G=N. $
  
  Assume  that $ G\neq N.$  Choose an element $ y\in Z(G)\setminus G'.$
Thus $ c\otimes yG' $ is a non-trivial element in $G'\otimes G/G'. $  By using Theorem \ref{lkk}, we can see that $ c\otimes yG' $ is non-trivial in $ \ker \eta $ and so
 \begin{align*}| \mathcal{M}(G)|p^{2}&\leq | \mathcal{M}(G)||G'||\ker \eta| \\&=|G/G'\wedge G/G'|| G'\otimes G/G'|\leq p^{\frac{1}{2}
 d(d-1)\alpha_1}p^{d}.\end{align*} Thus $
 | \mathcal{M}(G)|\leq  p^{\frac{1}{2}
 d(d-1)\alpha_1+d-2}.$
  Now, since $ n=\alpha_1 d+1, $ by our assumption we have $ | \mathcal{M}(G)|=p^{\frac{1}{2}(d-1)(n+1-2)+1} =p^{\frac{1}{2}d(d-1)\alpha_1+1}.$ It is a contradiction. 
   Now  $ G= NZ(G)=N.$ So, $ G $ is a capable minimal non-abelian $ p$-group. By  Lemma \ref{9}, Proposition \ref{min}, and \cite[Corollary 8.2]{3}, we get  \[G\cong \langle a, b\mid a^{p^m}= b^{p^m}=[ a, b]^p=1,[a,b,a]=[a,b,b]=1 \rangle,\] where  $m  > 1.$  
The converse holds by Lemma \ref{10}.
\end{proof}

\begin{prop}\label{m5}
Let $ G $ be a non-abelian  group of order $ p^n $ of  class $t  $ such that $ |G'|=p^k,$  $ d(G)=d $ and $|\mathcal{M}(G)|=p^{\frac{1}{2}(d-1)(n+k-2)+1}$ for all $ k $ such that $ k\geq 2.$ If
$K$ is a non-trivial central subgroup of order $p^m$ contained in $ Z(G)\cap G'\neq G',$
  then $|\mathcal{M}(G/K)|$ also attains the bound,  that is \[|\mathcal{M}(G/K)| =p^{\frac{1}{2}(d-1)(n + k - 2(m+1)) + 1}.\]
\end{prop}
\begin{proof}
 Let  $K$ be a non-trivial central subgroup of order $p^m$ contained in $ Z(G)\cap G'.$  We have $| G/K|=p^{n-m} $ and $ |(G/K)'|=p^{k-m}.$ We prove the result by using induction on $m.$ If $m=1,$ then the result holds by  Proposition \ref{m4}. Now let $ m\geq 2.$ Consider a normal subgroup $ K_1 $ in $ K $ of order $ p^{m-1}$ and using the   induction hypothesis, we have \[|\mathcal{M}(G/K_1)| =p^{\frac{1}{2}(d - 1)(n + k - 2m) + 1}.\]
Since $ K/K_1\subseteq Z(G/K_1)\cap  (G'/K_1)$ and $ |K/K_1| =p,$
  Proposition \ref{m4} implies that \[|\mathcal{M}(G/K)|=| \mathcal{M}(\dfrac{G/K_1}{K/K_1})| =p^{\frac{1}{2}(d - 1)(n + k - 2(m+1)) + 1}.\]
This completes the proof.
\end{proof}
The proof of the following corollary is similar to that of \cite[Theorem 1.2]{ri4}.
\begin{cor}\label{kh2}
Let $ G $ be a non-abelian  group of order $ p^n $ of  class $t\geq 3  $ such that $ |G'|=p^k,$  $ d(G)=d $ and $|\mathcal{M}(G)|=p^{\frac{1}{2}(d-1)(n+k-2)+1}$ for $ k\geq 2. $ Then $|\mathcal{M}( G/\gamma_i(G))| $ also attains  the bound  for all $ i $ such that $3 \leq i\leq t.$ 
\end{cor}
\begin{proof}
We prove the result by using induction on $ j=t-i+3 $ for all $ i $ such that $3 \leq i\leq t. $ If $ j=3, $ then since $ k\geq 2, $ by using  Proposition \ref{m5}, $|\mathcal{M}( G/\gamma_t(G))| $ attains  the bound. Using the   induction hypothesis,  $|\mathcal{M}( G/\gamma_i(G))| $ attains  the bound. 
\\
Since $ \gamma_{i-1}(G)/\gamma_i(G)\subseteq Z(G/\gamma_i(G))\cap G'/\gamma_i(G),$  Proposition \ref{m5} implies that  \[|\mathcal{M}( G/ \gamma_{i-1}(G))|=|\mathcal{M}(\dfrac{G/\gamma_i(G)}{\gamma_{i-1}(G)/\gamma_i(G)}  )| \] attains  the bound, as required.
\end{proof}

%\begin{lem}\label{kk}
%Let $ G $ be a finite $p$-group  of class two. Then $ e(G/Z(G))=e(G'). $
%\end{lem}
%\begin{proof}
%Straightforward.
%\end{proof}
\begin{prop}\label{fff}
Let $ G $ be a non-abelian  group of order $ p^n $  such that  $ G/G'\cong \mathbb{Z}_{p^{m}}^{(2)} $ and   $ G' \cong \mathbb{Z}_{p^{k}}$ with $ k\geq 2. $ Then
$|\mathcal{M}(G)|$ attains the bound if and only if
  \[G\cong \langle a, b\mid a^{p^m}= b^{p^m}=[ a, b]^{p^{m}}=1,[a,b,a]=[a,b,b]=1 \rangle,\]
where  $m\geq 2$ and $ p\neq 2 $ or
\[G\cong \langle a, b\mid a^{p^m}= b^{p^m}=[ a, b]^{p^{k}}=1,[a,b,a]=[a,b,b]=1,k\geq 2,m\geq 2 \rangle.\]
\end{prop}
\begin{proof}
Let $|\mathcal{M}(G)|$ attain the bound. Now by \cite[Theorem 1]{morse}, we get 
\begin{align*}
G&\cong G(\alpha,\beta,\gamma;\rho,\sigma)\\&=\langle a, b\mid [a,b]^{p^{\gamma}}=[a,b,a]=[a,b,b]=1,a^{p^{\alpha}}=[a,b]^{p^{\rho}},b^{p^{\beta}}=[a,b]^{p^{\sigma}}
\rangle,
\end{align*}
where $ \alpha+\beta+\gamma=n, $ $ 1\leq \gamma\leq \beta\leq \alpha, $ and $ 0\leq \rho,\sigma\leq \gamma.$ Since $ G/G'\cong \mathbb{Z}_{p^{m}}^{(2)}, $ we have $ \alpha=\beta. $ By Lemma \ref{ff} and \cite[Theorems 63 and 67$(i)$]{morse}, $ \rho=\sigma=\gamma.$ Using \cite[Theorem 1]{morse}, we conclude that 
 \[G\cong \langle a, b\mid a^{p^m}= b^{p^m}=[ a, b]^{p^{m}}=1,[a,b,a]=[a,b,b]=1 \rangle,\]
where  $m\geq 2$ and $ p\neq 2 $ or
\[G\cong \langle a, b\mid a^{p^m}= b^{p^m}=[ a, b]^{p^{k}}=1,[a,b,a]=[a,b,b]=1,k\geq 2,m\geq 2 \rangle.\] The converse holds by Lemma \ref{10}.
 \end{proof}
Now we are ready to obtain the structure of a $p$-group $ G $ of class two such that $ |\mathcal{M}(G)| $ attains the bound. 
\begin{lem}\label{120}
Let $ G $ be a non-abelian  group of order $ p^n $  of class two such that $ |G'|\geq p^2$ and $ d(G)=d. $ If $|\mathcal{M}(G)|$ attains the bound, then $2\leq d\leq 3.$
\end{lem}
\begin{proof}
If $ e(G/G')=p, $ then $ d=n-k. $ By using Theorem \ref{16},  $2\leq d\leq 3.$ Let $ e(G/G')>p. $ If $ k=1,  $ then Proposition \ref{m2} implies $ d=2. $ Now, assume that $ k\geq 2 $ and $ K\subsetneqq G' $ such that
$ |G'/K|=p. $ By Proposition \ref{m5}, $|\mathcal{M}(G/K)|$ attains the bound so $ d=d(G/K)=2, $ by Proposition \ref{m2}. Hence,  $2\leq d\leq 3. $
\end{proof}
\begin{lem}\label{m21}
There exists no   non-abelian  group $ G $  of order $ p^n $ of class two such that $ d(G)=3, $ $e( G')\geq p^{2},$ $ e(G/G')\geq p^{2},$ and  $ |\mathcal{M}(G)|=p^{\frac{1}{2}(d-1)(n+k-2)+1}.$ 
\end{lem}
\begin{proof}
Assume to the country that there is a such group $ G. $ Clearly,  $ e(G/Z(G))=e(G')= p^{k}.$ Using Lemma \ref{ll},  $d( G')\leq 3.$ Now let $G'\cong \mathbb{Z}_{p^{k}}.$
Consider the factor group $ G/G'^{p}. $
Obviously, $ (G/G'^{p})'\cong  \mathbb{Z}_{p} $ and $ d(G/G'^{p})=3. $ Using Proposition \ref{m5},
$|\mathcal{M}(G/G'^{p})|$ also attains the bound, so Proposition \ref{m2} implies 
$ d(G/G'^{p})=2. $ It is a contradiction. By a similar way, if
 $2\leq d( G')\leq 3, $ then we get a contradiction.
\end{proof}

\begin{thm}\label{m3}
 Let $ G $ be a non-abelian group of order $ p^n $  of class two such that $ |G'|=p^k$ and $ d(G)=d. $ Then
$|\mathcal{M}(G)|$ attains the bound  if and only if $ G $ is isomorphic to one of the following  groups:
\begin{itemize}
\item[$(i)  $]  For $ p\neq 2, $  $H_1\cong E_1\times \mathbb{Z}_{p}^{(n-3)}, $ where $ E_1 $ is the extra-special $p$-group of order $ p^3 $ and exponent $ p.$ 
\item[$(ii)  $]$H_2\cong \langle a, b\mid a^{p^m}= b^{p^m}=[ a, b]^p=1,[a,b,a]=[a,b,b]=1 \rangle,$ where  $m  > 1.$
\item[$(iii)  $]For $p\neq 2, $ $H_3\cong \langle a, b\mid a^{p^m}= b^{p^m}=[ a, b]^{p^{m}}=1,[a,b,a]=[a,b,b]=1 \rangle,$ where    $ m\geq 2.$
\item[$(iv)  $]$H_4\cong \langle a, b\mid a^{p^m}= b^{p^m}=[ a, b]^{p^{k}}=1,[a,b,a]=[a,b,b]=1, k\geq 2,m\geq 2 \rangle.$
 \item[$(v)  $]For $p\neq 2, $ $H_5\cong\mathbb{Z}_{p}^{(4)} \rtimes\mathbb{Z}_{p}.$ 
\item[$(vi)  $]For $p\neq 2, $ $ H_6\cong \langle x_1,x_2,x_3\mid [x_1,x_2]^p=[x_2,x_3]^p=[x_3,x_1]^p=x_i^p=1,$\[[x_1,x_2,x_i]=[x_3,x_1,x_i]=[x_2,x_3,x_i]=1, 1\leq i\leq 3\rangle.\] 
\end{itemize}
\end{thm}
\begin{proof}

Suppose that $|\mathcal{M}(G)|$ attains the bound.
First assume that $ |G'|=p.$ By Lemma \ref{m1}$(i),$ $ G/G' $ is homocyclic.  Theorem \ref{16}$(i)$ and Proposition \ref{m2} imply $ G\cong E_1\times \mathbb{Z}_{p}^{(n-3)}\cong H_1 $ or \[G\cong H_2\cong \langle a, b\mid a^{p^m}= b^{p^m}=[ a, b]^p=1,[a,b,a]=[a,b,b]=1 \rangle,\] where  $m  > 1.$
\\ Now suppose that $ |G'|\geq p^2.$  By using Lemma \ref{120},  $2\leq d\leq 3.$ Let $ d=2. $
 By Proposition \ref{fff}, we have \[G\cong \langle a, b\mid a^{p^m}= b^{p^m}=[ a, b]^{p^{m}}=1,[a,b,a]=[a,b,b]=1 \rangle,\]
where  $m\geq 2$ and $ p\neq 2 $ or
\[G\cong \langle a, b\mid a^{p^m}= b^{p^m}=[ a, b]^{p^{k}}=1,[a,b,a]=[a,b,b]=1,2\leq k\leq m \rangle.\] Thus $ G\cong H_3 $ or $ G\cong H_4. $
   \newline
Let now $ d=3. $ By Lemma \ref{m21},  $G/G'$ is of exponent $ p$ and hence $ G'=\phi(G).$ Thus $d(G)=n-k=3.$ Therefore \[|\mathcal{M}(G)|=p^{\frac{1}{2}(n-k-1)(n+k-2)+1}.\] Using Theorem \ref{16},  $G\cong H_5$ or $G\cong H_6.$ The converse  follows from Propositions \ref{m2}, \ref{fff} and Theorem \ref{16}. The proof is completed.
\end{proof}

\begin{prop}\label{two}
There exists no   non-abelian  $ 2$-generator group $ G $  of order $ p^n $ of class $ t\geq 3 $ such that $| G'|\geq p^2$ and  $ |\mathcal{M}(G)|=p^{\frac{1}{2}(d-1)(n+k-2)+1}$  with  $p\neq 2.$  
\end{prop}
\begin{proof}
Assume to the contrary that there is a such group $ G. $   Let $ G= \langle x,y\rangle,$ where   $ x,y\in G\setminus \phi(G).$ Without loss of generality, we may assume that  $ [x,y,x]\neq 1. $
So,    $\Psi_3(x G'\otimes y G'\otimes x G' \otimes  y G')=([[x,y],x]\otimes yG')^2([y,[x,y]]\otimes xG')^2 \neq 1.$ On the other hand, by Lemma \ref{m1}$(ii), $  $ |\mathrm{Im} \Psi_3|= 1. $ Hence we have a contradiction.
\end{proof}

We are ready to obtain the structures of $G$ when $ |\mathcal{M}(G)|=p^{\frac{1}{2}(d-1)(n+k-2)+1}.$
\begin{p b}
{\em Suppose that $G$ is  nilpotent of class two and $|\mathcal{M}(G)|$ attains the bound. Then $ G $ is isomorphic to one of the groups $  H_1, H_2,H_3, H_4,H_5$ or $ H_6, $ by  Theorem \ref{m3}. Now, let $G$ be nilpotent of class at least $3.$  Using Corollary \ref{kh2},
$|\mathcal{M}(G/\gamma_3(G))| $ attains the bound.  Theorem \ref{m3} implies that 
$ G/\gamma_3(G)\cong H_i $ for some $ i $ in $ 1\leq i\leq 6 $ and so 
$ 2\leq d(G/\gamma_3(G)) \leq 3.$  Since $ d(G)=d(G/\gamma_3(G)), $ we have $ 2\leq d(G)\leq 3. $ Thus by Proposition \ref{two}, $ d(G) = 3 $ and so $ d(G/\gamma_3(G))=3.$
By Theorem \ref{m3}, $ G/\gamma_3(G)\cong  H_5$ or $G/\gamma_3(G)\cong H_6.$ Hence $( G/\gamma_3(G))^{ab} $ is elementary abelian. Thus $ d(G)=n-k.$ By using  \cite[Theorem 1.2]{ri4}$(iv), $   $ G\cong H_7. $ The converse holds by  Theorem \ref{m3} and \cite[Theorem 1.2]{ri4}.}
\end{p b}

\end{document}